\documentclass{article}[11]
\usepackage{amsmath, amsthm, amssymb, color, amsrefs}

\newtheorem{thm}{Theorem}[section]
\newtheorem*{thm*}{Theorem}
\newtheorem{prop}[thm]{Proposition}

\newtheorem{cor}[thm]{Corollary}

\DeclareMathOperator\vol{vol}

\DeclareMathOperator\Ric{Ric}

\DeclareMathOperator\Vol{Vol}

\def\d{{\rm d}}

\DeclareMathOperator\tr{tr}
\def\C{\mathbb{C}}

\def\CP{\mathbb{CP}}
\def\Sph{\mathbb{S}}

\def\R{\mathbb{R}}
\def\Z{\mathbb{Z}}
\def\T{\mathbb{T}}

\begin{document}

\title{Pluri-subharmonic functions on complex tori, Ricci curvature and convexity}

\author{Tommaso Pacini\footnote{University of Torino, Italy,
\texttt{tommaso.pacini@unito.it}}}

%\date

\maketitle

\begin{abstract}
We show that, in toric manifolds, one can characterize the sign of the Ricci curvature in terms of the convexity of the volume functional. More generally we discuss relationships between (i) Ricci curvature and volume, (ii) totally real and Lagrangian submanifolds, (iii) pluri-subharmonic functions and convexity.
\end{abstract}

\section{Introduction}
The following result, proved in this paper, provides a good sample of the topics we are interested in.

\begin{thm*}
Let $M$ be a smooth toric variety endowed with a K\"ahler structure such that the $\T^n$-action is Hamiltonian. Consider the volume functional $\Vol$ restricted to the set of non-degenerate $\T^n$-orbits. Then:
\begin{itemize}
\item $\Ric<0$ iff $\log\Vol$ is strictly convex. This implies $\Vol$ is strictly convex. Such manifolds cannot be compact.
\item $\Ric>0$ iff $-\log\Vol$ is strictly convex. This implies $1/\Vol$ is strictly convex. If $M$ is compact then $\Vol$ has a unique critical point, corresponding to a global maximum. It is a minimal Lagrangian $n$-torus.
\end{itemize}
\end{thm*}
This statement summarizes the main contents of Section \ref{s:symplecticgeometry}. The assumptions regarding symmetries and strict curvature bounds can be relaxed, as in Section \nolinebreak \ref{s:vol_convexity}. 

From our point of view, an interesting aspect of this theorem is that it interweaves several different threads of geometry. The most obvious is the well-known relationship between Ricci curvature and volume: this is already apparent in Riemannian geometry, but it takes special form in the K\"ahler context via a certain relationship between the Ricci curvature and the anti-canonical bundle. A second, more novel, thread concerns a link between the anti-canonical bundle and submanifold geometry. More specifically, the theorem provides the perfect setting for reflecting on a string of relationships between:
\begin{enumerate}
\item Lagrangian and totally real (TR) submanifolds.
\item TR submanifolds and the anti-canonical line bundle $K_M^{-1}$. 
\item The curvature of $K_M^{-1}$ and pluri-subharmonic (PSH) functions.
\item PSH functions and convexity. 
\end{enumerate}
The concatenation of these elements, starting with Lagrangians and ending with convexity, can be taken as an outline of the proof of the theorem. 

We wish to emphasize the following aspects of this train of thought.

\paragraph{Complex tori and PSH functions.}The keystone of this network is yet another protagonist in geometry and beyond: the complex Lie group $(\C^*)^n$. This group can be seen from several viewpoints. It is the heart of toric geometry, in all its versions: complex, K\"ahler, symplectic. Covering maps relate it to $\C^n$ and to complex tori. By thinking of it as the complexification of the real torus $\T^n$, we can also view it as a $\T^n$-fibre bundle. 

Although all of the above will be relevant, the starting point for this paper is a more analytic viewpoint: specifically, a simple, abstract, correspondence between pluri-subharmonic (PSH) functions $g$ on $(\C^*)^n$ and convex functions $G$ on the base space of this $\T^n$-bundle. An important part of our reasoning is the observation that, in the context of K\"ahler geometry, this analytic correspondence admits a geometric reformulation. Indeed, let $g$ denote the potential of the Ricci 2-form on $(\C^*)^n$ determined by any K\"ahler metric $\omega$. The Ricci-volume relationship then implies that $g$ is closely related to the density function of the volume form, so that $G$ coincides with a certain volume functional on the space of $\T^n$-orbits. It follows that we can control the PSH condition for $g$ via the sign of the Ricci curvature, and in turn this controls the convexity of $G$.

\ 

\noindent\textit{Example. }The simplest example of this relationship occurs when $n=1$ and $\C^*$ compactifies to the toric variety $\CP^1$. The Ricci 2-form then coincides with the standard curvature form $k\omega=k\vol$, while $G$ encodes the length of the circles of constant latitude, ie those parallel to the equator. The condition $k>0$ corresponds to the obvious fact that $G$, apppropriately parametrized, is concave. The condition $k\leq 0$, impossible by Gauss-Bonnet, can be alternatively ruled out by the simple fact that it would imply that $G$ is a positive convex function which, thanks to the compactification, tends to zero at both extremes.

\paragraph{Complex vs symplectic data.}K\"ahler geometry blends three types of data: complex, symplectic and Riemannian. It is sometimes useful to give priority to one over the others. In our case it is thus important to notice that, for $n>1$, the $\T^n$-orbits in $(\C^*)^n$ are TR submanifolds. Classically, these submanifolds were of interest only as a tool in complex function theory, but our recent work, cf. eg \cite{LP1}, \cite{LP2}, has highlighted the existence of a specific geometry of TR submanifolds, including a canonical volume functional (distinct from the standard volume). It is this volume functional which is encoded by the function $G$, and our convexity results here are a concrete manifestation of the abstract picture presented in \cite{LP2}. The Lagrangian submanifolds discussed in the main theorem are a special case but the TR viewpoint, in the proof, is foundational to making the link with the theory of PSH functions.

\paragraph{Coordinate systems.}Convexity is one of the most efficient ways of proving existence/uniqueness results for critical points of a function. It also provides a clear geometric explanation for such results. Notice however an important discrepancy. Critical points are independent of a choice of coordinates. The notion of convexity depends instead on a preliminary choice of a family of parametrized curves. If the family arises as the lines defined by a system of coordinates, convexity will then depend strongly on that specific choice. This is already apparent in one dimension: the function $f(x):=e^x$ is no longer convex after changing coordinates via $x=\log(\log t)$ (restricted to $t>1$). 

The choice of complex vs symplectic (action-angle) coordinates, for toric manifolds, has already been discussed in the literature \cite{Abreu} and falls within the framework of the previous paragraph. Our arguments rely on complex cooordinate systems.

\ 

The paper is structured according to the following flow-chart:

\begin{itemize}
\item Section \ref{s:PSH} - Section \ref{s:vol_convexity} - Section \ref{s:symplecticgeometry}.

These introduce the relationship between PSH functions and convexity, then focus on PSH functions obtained as the potential of the curvature of $K_M^{-1}$ explaining the relationship with volume functionals, then prove the theorem.
\item Section \ref{s:complexanalysis} - Section \ref{s:linebundles} - Section \ref{s:toric}.

These can largely be taken as digressions of independent interest, which serve to enlarge the general picture.
\end{itemize}
\noindent\textit{Acknowledgements. }We wish to thank S. Diverio for an interesting conversation. The results contained in Sections \ref{s:PSH}, \ref{s:complexanalysis} are mostly classical, cf. also \cite{Demailly}, but it seems worthwhile to advertise them more widely.

\section{PSH functions and convexity}\label{s:PSH}
Let $J$ denote the standard complex multiplication by $i$ on $\C^n$. Recall that, for any $f:\C^n\rightarrow\R$ with regularity $C^2$, the 2-form $i\partial\bar\partial f$ is real-valued. We obtain a $J$-invariant symmetric 2-tensor by setting
\begin{equation*}
L_f[p](\cdot,\cdot):=i\partial\bar\partial f[p](\cdot,J\cdot).
\end{equation*}
Assume $f$ is pluri-subharmonic (PSH). This means that $i\partial\bar\partial f\geq 0$, in the sense that $L_f[p]$ is positive semi-definite, ie $L_f[p](v,v)\geq 0$, for all $p\in\C^n$ and $v\in\C^n$.

Let us write $z_i=x_i+iy_i$ and consider, at any point $p\in\C^n$, the basis $\{\partial x_1,\dots,\partial x_n,\partial y_1,\dots,\partial y_n\}$. The symmetric matrix associated to $L_f[p]$ can be written in the form
\begin{equation*}
L_f[p]=\frac{1}{2}\left(\begin{array}{cc}
A&B^t\\
B&A
\end{array}\right),
\end{equation*}
for appropriate matrices $A,B$ with $A$ symmetric. 

An important special case is when $f=f(x)$ is independent of the variables $y$. In this case $A=\mbox{Hess}_x(f)$ and $B=0$, so PSH is equivalent to convexity.

\ 

\noindent\textit{Example. }Let us illustrate the above in the case $n=2$, using the notation $(z,w)\in\C^2$ with $z=x+iy$ and $w=\xi+i\eta$. A straightforward computation then shows that, with respect to the basis $\{\partial x,\partial \xi,\partial y,\partial \eta\}$,
\begin{equation*}
L_f[p]=\frac{1}{2}\left(\begin{array}{cccc}
f_{xx}+f_{yy}&f_{x\xi}+f_{y\eta}&0&-f_{x\eta}+f_{y\xi}\\
f_{x\xi}+f_{y\eta}&f_{\xi\xi}+f_{\eta\eta}&f_{x\eta}-f_{y\xi}&0\\
0&f_{x\eta}-f_{y\xi}&f_{xx}+f_{yy}&f_{x\xi}+f_{y\eta}\\
-f_{x\eta}+f_{y\xi}&0&f_{x\xi}+f_{y\eta}&f_{\xi\xi}+f_{\eta\eta}
\end{array}\right).
\end{equation*}
This implies $\tr(L_f)=\Delta f$.

\ 

We now introduce a more general relationship between PSH and convexity. Assume that $f$ is periodic with respect to the $y$ variables, in the sense that it is well-defined on $\C^n/2\pi i\,\Z^n$. We will refer to the submanifolds $L=L(y)$ defined by the condition that each $x_i$ is constant as the \textit{canonical $n$-tori}. 

Consider the function $F:\R^n\rightarrow\R$ defined by integrating along these $n$-tori:
\begin{equation*}
F(x):=\frac{1}{(2\pi)^n}\iint f(x,y)\,dy.
\end{equation*}

\begin{prop}\label{prop:F} 
If $f$ is PSH on $\C^n/2\pi i\,\Z^n$ then $F$ is convex.
\end{prop}
\begin{proof}
The point is that each term in $A-\mbox{Hess}_x(f)$ contains a $y$-derivative, so it disappears under integration. It follows that
\begin{equation*}
\mbox{Hess}(F)=\frac{1}{(2\pi)^n}\iint \mbox{Hess}_x(f)\,dy=\frac{1}{(2\pi)^n}\iint A\,dy.
\end{equation*}
Since restriction preserves semi-positivity, $A$ is semi-positive with respect to $x$-directions so $\mbox{Hess}(F)\geq 0$.
\end{proof}

Notice that when $f=f(x)$, ie $f$ is the pull-back of some $F$, we recover the simpler equivalence of PSH and convexity. 

Roughly speaking, this construction defines a surjective linear map
\begin{equation*}
\{PSH\}\rightarrow\{Convex\}, \ \ f\mapsto F
\end{equation*}
with kernel given by the subspace of PSH functions $f$ which integrate to zero along each canonical $n$-torus and a right-inverse map given by pull-back. We will obtain a better description of this kernel in the next section.

\ 

The construction can also be reformulated as follows. Given a function $g:(\C^*)^n\rightarrow\R$, consider the integral average function $G:(0,\infty)^n\rightarrow\R$ defined by integrating along the canonical $n$-tori $L=L(\theta)$ defined by the condition that each $r_i$ is constant:
\begin{align*}
G(r)&:=\frac{1}{\Pi_{i=1}^n (2\pi r_i)}\iint g(r_1e^{i\theta_1},\dots,r_ne^{i\theta_n})d\sigma\\
&=\frac{1}{(2\pi)^n}\iint g(r_1e^{i\theta_1},\dots,r_ne^{i\theta_n})d\theta,
\end{align*}
where $d\sigma$ is the natural volume element on the $n$-torus of radius $(r_1,\dots,r_n)$.

\begin{prop}\label{prop:G} 
If $g$ is PSH on $(\C^*)^n$ then $G$ is convex with respect to the variables $\log r_i$, ie $G\circ\exp:\R^n\rightarrow\R$ is convex with respect to the variable $x\in\R^n$.
\end{prop}

\begin{proof} The point is that the PSH condition is invariant under biholomorphisms. The exponential map gives a biholomorphism 
\begin{equation*}
\exp:\C^n/2\pi i\,\Z^n\rightarrow(\C^*)^n,\ \ (x,y)\mapsto (e^{x_1+iy_1},\dots,e^{x_n+iy_n})
\end{equation*}
Set $f:=g\circ \exp$, ie $f(x,y):=g(e^{x_1}e^{iy_1},\dots,e^{x_n}e^{iy_n})$. Then
\begin{align*}
F(x)&=\frac{1}{(2\pi)^n}\iint f(x,y)\,dy=\frac{1}{(2\pi)^n}\iint g(e^{x_1}e^{iy_1},\dots,e^{x_n}e^{iy_n})\,dy\\
&=\frac{1}{(2\pi)^n}\iint g(r_1e^{i\theta_1},\dots,r_ne^{i\theta_n})\,d\theta=G(r_1,\dots,r_n)=G\circ\exp(x),
\end{align*}
 where we set $r_i=e^{x_i}$ and $\theta_i=y_i$. The result follows from Proposition \ref{prop:F}.
\end{proof}

Of course, if $-f$ (or $-g$) is PSH then $F$ (or $G$) is concave.

\ 

\noindent\textit{Remark. }The identification $\R^n\simeq (0,\infty)^n$ induced by the change of variables $x=\log r$ transforms lines (along which $F$ is convex) into a geometrically very different system of curves. For example, when $n=2$, the line $x_1+x_2=0$ is transformed into the curve $\log(r_1r_2)=0$, thus into the hyperbola $r_1r_2=1$.

\ 

\noindent\textit{Remark. }The canonical $n$-tori coincide with the orbits of the obvious action of the Lie group $\T^n$ on $(\C^*)^n$. It is clear that the above results extend to any \textit{Reinhardt subdomain} of $(\C^*)^n$, ie invariant under this action. For example it extends to products of annuli or punctured disks.

\ 

Proposition \ref{prop:G} generalizes a classical result of Riesz\footnote{or is it Hardy?} concerning the behaviour with respect to $r$ of the integral average along circles of a subharmonic function $g$ defined on a planar annulus. An alternative proof relies on comparing $g$, on any annulus, with a harmonic function with the same values on the boundary of the annulus and using the divergence theorem. 

To conclude, notice that replacing a pointwise quantity $g$ with an integral quantity $G$ clearly implies a loss of information. This corresponds to the kernel of the map $g\mapsto G$. On the other hand it substitutes the PSH condition, thus the rather sophisticated pluri-potential theory, with the much simpler convexity condition, thus elementary calculus. This is analogous to the general philosophy behind the construction of invariants: extract simplified information from a given geometric structure, which is easier to manipulate and sufficient to determine certain aspects of the geometry. We will further investigate this in the next section, then apply it to data generated from a K\"ahler structure.

\section{Applications to pluri-potential theory}\label{s:complexanalysis}

In order to put the results of Section \ref{s:PSH} into context, recall that a function $u$ on $\R^n$ is subharmonic if $\Delta u\geq 0$. The Mean Value Theorem (MVT) provides a fundamental tool for studying this condition by integrating $u$ over spheres $\Sph^{n-1}(r)\subseteq\R^n$ of radius $r$. The starting point for the theory is then a monotonicity result for the function $r\mapsto U(r):=\frac{1}{\omega}\iint u \frac{d\sigma}{r^{2n-1}}$, where $\omega$ is the volume of $\Sph^{n-1}(1)$: if $u$ is subharmonic this function is non-decreasing and $\lim_{r\rightarrow 0} U(r)=u(0)$. Furthermore, $U\equiv 0$ if and only if $u$ is harmonic and $u(0)=0$. In other words we have defined a linear map 
\begin{equation*}
\textit\{Subharmonic\ functions\}\rightarrow\textit\{Monotone\ functions\} 
\end{equation*}
and characterized its kernel. This map is of great use in studing the properties of subharmonic functions. 

\ 

In the context of subharmonic functions the case $n=2$ stands out as somewhat exceptional, as in Riesz' convexity theorem. Section \ref{s:PSH} serves to emphasize that this case should instead be viewed as the lowest-dimensional instance of pluri-potential theory.

The fact that any PSH function on $\C^n\simeq\R^{2n}$ is subharmonic implies that theorems concerning subharmonic functions apply automatically to PSH functions. Section \ref{s:PSH} suggests however that, in this context, tori offer an interesting alternative to spheres when working in the domain $(\C^*)^n$, encoding the specificities of the PSH condition. Furthermore, replacing codimension 1 with codimension $n$ submanifolds reduces the loss of information implicit in any MVT. In later sections we shall think of the PSH function as a volume density for the $n$-tori, so this toric analogue of the MVT will take on a very geometric flavour.

We collect here several applications of this construction. The proofs emphasize the use of integrals over tori and of standard calculus (including the divergence theorem), avoiding the more sophisticated use of harmonic functions built via the Perron method.

\paragraph{Monotonicity.}
As usual, it is convenient to start with the following monotonicity result.

\begin{prop}\label{prop:monotonicity}
Let $g$ be a PSH function defined on a Reinhardt domain in $(\C^*)^n$. If $g$ extends to a PSH function for $z_i=0$ then $G$ is non-decreasing with respect to the variable $r_i$.

It follows that if $g$ extends to a PSH function on a neighbourhood of $0\in\C^n$, then $G$ is non-decreasing with respect to each variable $r_i$. Furthermore, $G\equiv c$ is constant if and only if $g$ is harmonic with respect to each coordinate $z_i$ separately and $g(0)=c$.
\end{prop}
\begin{proof}For simplicity let us restrict to the case $n=2$ and assume $g$ is PSH on $\C\times\C^*$, thus subharmonic on any complex line defined by fixing $z_2:=w$. Then, by the divergence theorem,
\begin{align*}
\frac{d}{dr}\int_{\Sph^1} g(re^{i\theta_1},w)\,d\theta_1&=\int_{\Sph^1} \frac{\partial}{\partial r} g(re^{i\theta_1},w)\,d\theta_1=\frac{1}{r}\int_{\Sph^1}\frac{\partial}{\partial r} g(re^{i\theta_1},w)\, d\sigma\\
&=\frac{1}{r}\iint \Delta g(x+iy,w)\,dxdy\geq 0.
\end{align*}
This proves that $\int_{\Sph^1} g(re^{i\theta_1},w)\,d\theta_1\leq \int_{\Sph^1} g(Re^{i\theta_1},w)\,d\theta_1$, for any $r\leq R$ and $w=r_2e^{i\theta_2}$. Integrating with respect to $\theta_2$ leads to $G(r,r_2)\leq G(R,r_2)$ for any $r_2$, as desired.

Now assume $G$ is constant. Then, as before and assuming $r_1>0$ and $r_2>0$,
\begin{align*}
0&\equiv\frac{\partial G}{\partial r_1}(r_1,r_2)=\int_{\Sph^1}\left(\int_{\Sph^1}\frac{\partial g}{\partial r_1}d\theta_1\right) d\theta_2=\int_{\Sph^1}\left(\frac{1}{r_1}\int_{\Sph^1}\frac{\partial g}{\partial r_1}d\sigma\right) d\theta_2\\
&=\int_{\Sph^1}\left(\frac{1}{r_1}\iint\Delta g(x+iy,r_2e^{i\theta_2})\,dxdy\right) d\theta_2.
\end{align*} 
Since $\Delta g\geq 0$ this implies $\Delta g=0$ with respect to the variable $z_1$. The same holds with respect to $r_2$. By continuity the conclusion holds also for $r_1=0$ or $r_2=0$. The value $g(0)=c$ arises from taking the limits $r_i\rightarrow 0$. The converse is similar. 
\end{proof}

Notice here the relevance of the product structure of $(\C^*)^n$. Restricting to each variable separately would be less natural in the context of Proposition \ref{prop:F}, where it is important that convexity holds for all directions.

It follows from Proposition \ref{prop:monotonicity} that the kernel of the map $\{PSH\}\rightarrow\{convex\}$, on polydisks, is given by the PSH functions $g$ which are harmonic with respect to each $z_i$ separately and such that $g(0)=0$. In particular such functions are harmonic.

\paragraph{Interior maximum principles.}Monotonicity of integrals over spheres implies an interior maximum principle. We can modify the usual proof, using integrals over tori, as follows. Assume $g$ is PSH on $\C^n$ and has a maximum in $0$. By Proposition \ref{prop:monotonicity} it must be locally constant on all $n$-tori in $(\C^*)^n$. By continuity it is thus constant in a neighbourhood of $0\in\C^n$. A standard open/closed argument then implies it is constant.

The function $G$ (or $F$) also satisfies a maximum principle. The following statement is an immediate consequence of convexity.

\begin{prop}\label{prop:maxprinciple}
Let $g$ be PSH on a Reinhardt domain in $(\C^*)^n$. If the corresponding function $G$ has a maximum point, it is constant.
\end{prop}
\paragraph{Hadamard's 3-circle theorem.}The classical version of this theorem implies that, if $g$ is subharmonic in an annulus in $\R^2$, then the function 
\begin{equation*}M(r):=\max\{g(re^{i\theta}):\theta\in\Sph^1\}
\end{equation*} 
is convex with respect to $\log r$. We can extend this to tori as follows.

\begin{prop}\label{prop:Hadamard}
Let $g$ be PSH on a Reinhardt domain in $(\C^*)^n$. Then the function
\begin{equation*}
M(r_1,\dots,r_n):=\max\{g(r_1e^{i\theta_1},\dots,r_ne^{i\theta_n}):(\theta_1,\dots,\theta_n)\in\T^n\}
\end{equation*}
is convex with respect to the variables $\log r_i$.
\end{prop}
\begin{proof}
Since $\T^n$ acts holomorphically, each pull-back function $\theta^*g$ is PSH so, by standard theory, the supremum of all such functions is again PSH. This function coincides with $M$. It follows that $M$ is PSH and $\theta$-independent, thus convex.
\end{proof}

More generally, the properties of $M$ are similar to those of the function $G$. Here is another example.

\begin{prop}\label{prop:monotonicity}
Let $g$ be a PSH function defined on a Reinhardt domain in $(\C^*)^n$. If $g$ extends to a PSH function for $z_i=0$ then $M$ is non-decreasing with respect to the variable $r_i$.

If $g$ extends to a PSH function on a neighbourhood of $0\in\C^n$ then $M$ is non-decreasing with respect to each variable $r_i$.
\end{prop}
\begin{proof}As before, we will restrict ourselves to the simplest case where $g$ is defined on $\C\times\C^*$. Fix $z_2:=w$ and set $M'(r_1,w):=\max\{g(r_1e^{i\theta_1}, w):\theta_1\in\Sph^1\}$. The maximum principle then implies that $M'$ is monotone with respect to $r_1$. Since this holds for any $w=r_2e^{i\theta_2}$ and $M(r_1,r_2)=\max\{M'(r_1,r_2e^{i\theta_2}): \theta_2\in\Sph^1\}$, the statement is clear.
\end{proof}

The following strong maximum principle is now clear.
\begin{prop}\label{prop:interiormax}
Let $g$ be PSH on a polydisk $\prod\{|z_i|<r_i\}\subseteq\C^n$ and continuous up to the boundary. The maximum of $g$ is then achieved on the distinguished boundary, ie the $n$-torus $\prod\{|z_i|=r_i\}$.
\end{prop}
\paragraph{Liouville's theorem.} Proposition \ref{prop:Hadamard} leads to the following version of Liouville's theorem.
\begin{prop}\label{prop:Liouville}
Any PSH function $g$ on $(\C^*)^n$ bounded from above is constant. 
\end{prop}
\begin{proof}
The corresponding function $M$ is convex on $\R^n$ (using the variables $\log r_i$) and bounded from above, thus constant. Any $n$-torus achieves this maximum value, so by the maximum principle $g$ is constant.
\end{proof}

An alternative proof is based on pulling $g$ back to $\C^n$ and applying the standard Liouville theorem for PSH functions.

\ 

Geometrically, Proposition \ref{prop:Liouville} has the following consequence. Let $(M,J)$ be a complex manifold endowed with a PSH function $g$. Assume given a holomorphic map $\phi:(\C^*)^k\rightarrow M$ such that $g\circ\phi$ tends to zero. Roughly speaking, this means that ``at infinity" the image of $(\C^*)^k$ lies in the zero set $Z(g):=g^{-1}(0)$. The conclusion is that the whole image lies in $Z(g)$. 

\ 

\noindent\textit{Remark. }Holomorphic maps $\phi:(\C^*)^k\rightarrow M$ are relevant in at least two contexts. 

First, $(\C^*)^k$ is a complex Lie group. We may think of it as the complexification of $\T^k$. The action of such Lie groups on complex manifolds is an important research topic, including for example the theory of toric varieties. Given any such action, any point $p\in M$ yields a holomorphic map $\phi:(\C^*)^k\rightarrow M$ defined by $g\mapsto g(p)$. Its image is simply the orbit of $p$.

Second, any such $\phi$ lifts to a holomorphic map $\hat\phi:\C^k\rightarrow M$, then restricts to a holomorphic map defined on any complex line $\C\subseteq\C^k$. Conversely, any holomorphic map $\C\rightarrow M$ restricts to a map defined on $\C^*$. The existence of (non-constant) holomorphic maps $\C\rightarrow M$ is in turn closely related to important complex-analytic properties of $M$. In particular, recall that Kobayashi defined a pseudo-distance $d$ on any complex manifold such that any holomorphic map is a contraction. When $M=\C$, $d\equiv 0$. A manifold $M$ is \textit{hyperbolic} if $d$ is actually a distance: this implies that any holomorphic map $\C\rightarrow M$ is constant. Actually, if $M$ is compact, Brody \cite{Brody} proved that such maps exist if and only if $M$ is not hyperbolic.

We will return to both of these contexts in Section \ref{s:toric}.

\section{Ricci curvature and volume}\label{s:linebundles}

The PSH functions we are interested in arise from the following geometric construction.

Any Hermitian metric $h$ on a holomorphic vector bundle $E\rightarrow M$ induces a canonical connnection on $E$ known as the \textit{Chern connection}. Its curvature $\Theta$ is a global $End(E)$-valued (1,1)-form on $M$. We shall restrict our attention to the case where $E=L$ is a holomorphic line bundle. In this case the Chern connection has a convenient expression in terms of any local nowhere-vanishing holomorphic section $\sigma$ of $L$. Set $H:=h(\sigma,\sigma)$. The connection is then described by the 1-form $\partial \log H$, while $\Theta=\bar\partial\partial \log H$. The object
\begin{equation*}
\rho(X,Y)=i\Theta(X,Y)=i\partial\bar\partial (-\log H)(X,Y)
\end{equation*} 
is a closed global real-valued (1,1)-form which, by Chern-Weil theory, represents the (real) first Chern class of $L$: $[\rho]=2\pi c_1(L)$. Thus $\rho$ is positive (negative) semi-definite if and only if $-\log H$ ($\log H$) is PSH. In other words, we can control the PSH condition using the curvature of $L$, thus $c_1(L)$. 

This has for example the following application.

\begin{prop}\label{prop:nonpositive_c1}
Let $M$ be a compact complex manifold which satisfies the $\partial\bar\partial$-lemma (for example, $M$ K\"ahler). Let $L\rightarrow M$ be a non-trivial holomorphic line bundle. If $c_1(L)\leq 0$ then $L$ does not admit non-trivial holomorphic sections. 
\end{prop}
\begin{proof}The condition $c_1(L)\leq 0$, together with the $\partial\bar\partial$-lemma, implies that $L$ has non-positive curvature for some Hermitian metric $h$. Assume $\sigma$ is a global holomorphic section. Since $L$ is non-trivial, $\sigma$ must vanish somewhere so its zero set $D$ is non-empty. Set $H:=h(\sigma, \sigma)$. Then $\log H$, thus $H$, is PSH. Since $h$ is globally well-defined, $H$ is constant by the max principle. Since $H\rightarrow 0$ near $D$, $H\equiv 0$ so $\sigma\equiv 0$.
\end{proof}

In particular, let $K_M:=\Lambda^{n,0}M$ denote the \textit{canonical} line bundle on $M$. We shall apply the above construction to its dual, ie the \textit{anti-canonical} line bundle $L:=K_M^{-1}$. It is customary to write $c_1(M):=c_1(K_M^{-1})$. In this setting, one way to obtain $h$ and $\sigma$ as above is by starting with appropriate data on $M$. Indeed, any Hermitian metric $h$ on $M$ induces a Hermitian metric $h$ on $K_M^{-1}$, and linearly independent local holomorphic vector fields $v_1,\dots,v_n$ on $M$ (such as $v_i:=\partial z_i$ generated by coordinates) define a holomorphic section $\sigma:=v_1\wedge\dots\wedge v_n$ of $K_M^{-1}$. With these choices $H=\det h_{i \bar j}$, where $h_{i \bar j}:=h(v_i,v_j)$. 

Now assume $M$ is K\"ahler. In this context it is well-known that the Chern connection on $K_M^{-1}$ coincides with the Levi-Civita connection and that $\rho(\cdot,J\cdot)=\Ric$, where the latter denotes the standard (Riemannian) Ricci curvature of $M$. Furthermore, the  formula $h=g-i\omega$ relates the Hermitian metric on $M$ to the Riemannian and symplectic structures. The vectors $v_1,\dots,v_n,Jv_1,\dots,Jv_n$ provide a real local basis, so we can write the Riemannian volume form $\vol_M$ in terms of (the square root of) the determinant of their dot products, thus in terms of $g(v_i,v_j)$ and $\omega(v_i,v_j)$. It follows that $\det h_{i\bar j}$ coincides with the density of this volume form. 

To summarize: the same quantity $H$ appears in both formulae 
$$\rho=i\partial\bar\partial(-\log H),\ \ \vol_M=H\,v_1^*\wedge\dots\wedge v_n^*\wedge Jv_1^*\wedge\dots\wedge Jv_n^*$$ 
for the Ricci 2-form and the ambient volume form.

\ 

\noindent\textit{Remark. }This is a special manifestation of the close relationship between Ricci curvature and volume. Another is the well-known fact that, in any Riemannian manifold, the Ricci curvature appears as one of the coefficients of the Taylor expansion of the density of the volume form, along any geodesic.

\section{Digression: toric varieties.}\label{s:toric}

Toric varieties provide a perfect geometric setting in which to apply these constructions. Although the results in this section can often be proved in other ways, it seems worthwhile to emphasize their relationship with this theory. 

Recall that a toric variety is a manifold $M$ endowed with an effective $G:=(\C^*)^n$-action admitting a dense open orbit. The fact that $G$ is Abelian implies that the isotropy groups $G_p$ of any orbit are independent of the specific point $p$, and by continuity the isotropy group $G'$ of the open dense orbit is contained in the isotropy group of any other orbit, ie $G'\leq G_p$ for all $p\in M$. The effectivity condition thus implies that $G'$ is trivial, so the action of $G$ on the open dense orbit is free. If $M$ is compact, we can thus think of it as a compactification of $(\C^*)^n$. 

\paragraph{Topology.}The following result is a variation on Proposition \ref{prop:nonpositive_c1}.

\begin{prop}\label{prop:nonpositive_toric}
The first Chern class $c_1(M)$ of a compact K\"ahler toric manifold $M$ cannot satisfy $c_1(M)\leq 0$. In particular, $M$ cannot admit a K\"ahler metric with $\Ric\leq 0$.
\end{prop}
\begin{proof}Let $(\C^*)^n\subseteq M$ be the open dense orbit, endowed with the induced metric. The vector fields $v_i:=\partial y_i$ defined by the $\T^n$-action are globally well-defined and holomorphic but, since the orbit's closure contains only lower-dimensional orbits, they cannot be linearly independent outside of this orbit. It follows that the holomorphic section $\sigma:=v_1^*\wedge\dots\wedge v_n^*$ of $K_M^{-1}$ vanishes outside this orbit. Assuming $c_1(M)\leq 0$, let us choose $h$ as in the proof of Proposition \ref{prop:nonpositive_c1}. The corresponding function $H:=h(\sigma,\sigma)$ is globally well-defined, PSH, and zero outside this orbit. We may conclude as in Proposition \ref{prop:nonpositive_c1}. 

Alternatively, the function $G$ obtained by integrating $H$ along the $\T^n$-orbits is convex on $\R^n$. Since $H$ is bounded, the dominated convergence theorem shows that $G\rightarrow 0$ near the boundary of the orbit, so it is constantly zero. This constradicts the fact that $H$ is strictly positive on $(\C^*)^n$.
\end{proof}

The proof shows that the union of the lower-dimensional orbits are a divisor for $-K_M$. 

We can think of Proposition \ref{prop:nonpositive_toric} as a vast generalization of the Gauss-Bonnet theorem applied to the sphere $\Sph^2$: indeed, $\CP^1$ is a toric variety. Both Gauss-Bonnet and the proof above encode, in different ways, the Chern-Weil point of view on the characteristic class $c_1(M)$.

\ 

\noindent\textit{Remark. }The alternative argument in the proof is very much in the same spirit as Maccheroni's non-filling results \cite{Maccheroni} for minimal Lagrangian tori. In her case the contradiction concerned the possibility of interpolating between a zero and a critical point of a non-negative convex functional.

One can alternatively use the theory of holomorphic vector fields to show that a compact complex manifold with $c_1(M)<0$ has a discrete group of complex automorphisms.

\ 

\noindent\textit{Remark. }Similar facts hold for complex tori $M:=\C^n/\Lambda$. Any such torus admits K\"ahler metrics. It is clear that $c_1(M)=0$, so $\int_M c_1(M)^n=0$. It follows that any K\"ahler metric with non-positive (or non-negative) Ricci curvature must satisfy $\Ric\equiv 0$. We can relate this to our convexity results by changing coordinates so that the metric lifts to $\C^n/2\pi i \Z^n$. Integration along $\T^n$-orbits then leads to a function $G$ which is periodic and convex, thus constant. On the other hand, if for example $\Ric[p](v,v)<0$, the symmetries of $\Ric$ would imply that $\Ric$ is negative on the complex line generated by $v$, leading to a direction transverse to the orbits along which $G$ is strictly convex.

\paragraph{Immersions.}We have already discussed the fact that hyperbolic manifolds do not contain complex lines. It follows that they do not admit immersions of toric varieties or of complex tori.

A special class of hyperbolic manifolds is provided by the Schwarz-Pick-Ahlfors lemma, which shows that any manifold $M$ with negative holomorphic sectional curvature is hyperbolic. Usually, the relationships between the curvatures of a manifold and those of its submanifolds are highly non-trivial, but in this case the situation is different, allowing us to re-formulate non-immersion results in terms of curvature and convexity, as follows. 
  
\begin{prop}\label{prop:holsubs}
Let $M$ be a projective manifold. Assume it admits a K\"ahler metric with negative holomorphic sectional curvature. Then it does not admit holomorphic immersions of compact toric varieties or of complex tori.
\end{prop} 
\begin{proof}The curvature-decreasing property of complex submanifolds implies that any such submanifold would have negative holomorphic sectional curvature. It would also be projective, so a result by Wu and Yau \cite{WuYau} then implies that it would admit a negative K\"ahler-Einstein metric, contradicting the above results.
\end{proof}

\noindent\textit{Remark. }Rational curves are a special case. Let $M$ be a complex (not necessarily projective) manifold with non-positive holomorphic sectional curvature. Any holomorphic immersion $\CP^1\rightarrow M$ would contradict Gauss-Bonnet. More generally, our arguments above apply to any holomorphic immersion $\C^*\rightarrow M$ which compactifies to a holomorphic map $\CP^1\rightarrow M$. 

Alternatively, under the stronger assumption that all sectional curvatures are non-positive, the Cartan-Hadamard theorem implies $M$ is aspherical so any rational curve would be homologically trivial: this would contradict the fact that holomorphic submanifolds are area-minimizing in their homology class.

The paper \cite{HLW}, Theorem 1.4, proves similar results without assuming the map is an immersion, but assuming $M$ is projective. 

\paragraph{$\C^*$-actions.}The notion of toric varieties can be extended to that of $T$-varieties, defined by the action of some $(\C^*)^k$. We shall concentrate only on the extreme case $k=1$. 

Assume $\C^*$ acts holomorphically on a complex manifold $M$. Choose any point $p\in M$ and consider the map $\C^*\rightarrow M$, $g\mapsto g\cdot p$. Either the infinitesimal action is trivial at $p$, in which case $p$ is a fixed point, or it is not: in this case  $p$ is a regular value and the map is an immersion, because $\C^*$ is Abelian so the isotropy groups coincide for all points of the orbit. 

Assume $M$ is compact and admits fixed points. Then Sommese \cite{Sommese} proved that every $\C^*$-orbit compactifies to a holomorphic map $\CP^1\subseteq M$. As in the previous remark, it follows that $M$ does not admit a K\"ahler metric with non-positive holomorphic sectional curvature. We summarize as follows.

\begin{prop}Let $M$ be a compact complex manifold. The following conditions are mutually exclusive.
\begin{enumerate}
\item $M$ admits a (non-trivial) holomorphic $\C^*$-action with fixed points.
\item $M$ admits a K\"ahler metric with non-positive holomorphic sectional curvature.
\end{enumerate} 
\end{prop}

\noindent\textit{Example. }Assume given a $\T^1$-action on a compact complex manifold. Let $\tilde X$ denote a fundamental vector field. Then $J\tilde X$ defines a complete vector field, allowing us to complexify the $\T^1$-action, thus obtaining a $\C^*$-action. 

In this context two main types of situations may arise. Assume for example $M:=\C/\Lambda$ with $\Lambda:=<1,\tau>$. Notice that $\partial x$ generates a free $\T^1$-action. If $\tau=i$, the complexified action is transitive. Indeed, it coincides with the action of the torus on itself. Alternatively, choosing $\tau$ appropriately, we may assume $\partial y$ corresponds to a free $\R$-action with a dense orbit. We thus obtain a transitive $\C^*$-action whose isotropy group $H$ is not closed, so there is no relation between the topology of the torus and the topology of the quotient $\C^*/H$. 

In both cases the action is non-effective. However, it is fixed-point free so Sommese's result does not apply, neither action compactifies to $\CP^1$ and indeed the torus admits flat metrics (thus non-positive curvature).

\ 

\noindent\textit{Remark. }One of the simplest ways of ensuring the existence of fixed points is by starting with a Hamiltonian $\T^1$-action, for some K\"ahler structure. Compactness implies that this action extends to a $\C^*$-action, and any critical point of the Hamiltonian function will generate a fixed point. More generally, Gilligan et al. \cite{Gilligan} show that the closure of the complexified orbits of any Hamiltonian action contains only orbits of strictly lower dimension, as in the case of toric varieties.

\section{Convexity of volume functionals}\label{s:vol_convexity}

Section \ref{s:PSH} admits a nice geometric interpretation in terms of volume functionals and the Ricci tensor. Explaining this will require a few background notions.

\paragraph{Totally real geometry.}Each canonical $n$-torus $L\subseteq (\C^*)^n$ is an example of a \textit{totally real} (TR) submanifold, defined by the condition $T_pL\cap J(T_pL)=\{0\}$. In dimension $n$ it follows that $T_pL\oplus J(T_pL)=T_pM$, so any real basis $\{v_1,\dots,v_n\}$ of $T_pL$ is also a complex basis of $T_pM$. 

A key feature of $n$-dimensional TR submanifolds is that orientability implies that $K_{M|L}$ is trivial: indeed, any non-vanishing section of $\Lambda^n(L)$ complexifies to a non-vanishing section of $\Lambda^{n,0}M_{|L}$. This suggests the possibility of a link between TR geometry and the geometry of $K_M$. This idea was developed in \cite{LP1}, \cite{LP2}. Here we will assume that $L$ is an oriented TR submanifold in a complex manifold $M$ endowed with a Hermitian metric $h$ on $K_M^*$. Given $p\in L$, $\{v_1,\dots,v_n\}$ will denote a positive basis of $T_pL$.
 
The starting point is the simple observation that the standard volume form $\vol_g:=\frac{v_1^*\wedge\dots\wedge v_n^*}{|v_1^*\wedge\dots\wedge v_n^*|_g}$ on $L$ does not make use of the TR condition. One should thus not expect the standard volume functional to have special properties on TR submanifolds. If however we complexify each dual element $v_i^*$ to become an element in $\Lambda^{1,0}_pM$, we can define the $(n,0)$-form on $L$
\begin{equation*}
\Omega_J:=\frac{v_1^*\wedge\dots\wedge v_n^*}{|v_1^*\wedge\dots\wedge v_n^*|_h}=|v_1\wedge\dots\wedge v_n|_h\cdot (v_1^*\wedge\dots\wedge v_n^*)\in {K_M|L},
\end{equation*}
then restrict it to a real-valued $n$-form $\vol_J:=\Omega_{J|TL}\in\Lambda^n(L)$. The expression above is independent of the basis. More invariantly, $\Omega_J=\sigma/|\sigma|_h$, where $\sigma$ denotes (the complexification of) any real volume form on $L$.

The corresponding \textit{$J$-volume} functional is $\Vol_J(L):=\int_L\vol_J$. 

Notice that we have defined this functional only for $n$-dimensional oriented totally real submanifolds. Here, it is strictly positive. However, the definition shows that it admits a continuous extension to zero when the vectors $v_i$ tend towards becoming $\C$-linearly dependent, ie when the limiting submanifold contains complex directions or has lower dimension.

When $h$ is induced from $M$ then $|v_1\wedge\dots\wedge v_n|^2_h=\det h(v_i,v_j)$ so
\begin{equation*}
\Omega_J=\sqrt{\det h(v_i,v_j)}\cdot(v_1^*\wedge\dots\wedge v_n^*).
\end{equation*}

Now assume $M$ is K\"ahler, so that $h=g-i\omega$. We then encounter a special case of the TR condition: the $n$-dimensional \textit{Lagrangian submanifolds}, defined by condition $\omega_{|TL}\equiv 0$, equivalently $T_pL\perp J(T_pL)$. In this case $h(v_i,v_j)=g(v_i,v_j)$, so $\vol_J=\vol_g$ and $\Vol_J$ coincides with the standard volume functional. For $n=1$, ie curves in Riemann surfaces, this condition is automatically satisfied.

\ 

\noindent\textit{Remark. }The Lagrangian condition has many interesting features, and we will come back to it in Section \ref{s:symplecticgeometry}. However, it is too restrictive to provide an appropriate framework for a geometric interpretation of Section \ref{s:PSH}. We will see below that totally real submanifolds and the $J$-volume fit the need perfectly.

\paragraph{Convexity of volume functionals.}We can now reformulate Section \ref{s:PSH} in geometric terms. 

Set $M:=\C^n/2\pi i\Z^n\simeq (\C^*)^n$. Thinking of it as a complex Lie group shows that the anti-canonical bundle is holomorphically trivial; a section $\sigma$ is generated by the vector fields $v_i:=\partial y_i\simeq\partial\theta_i$, which are holomorphic as they correspond to the holomorphic action of $\T^n$ on $(\C^*)^n$. Any metric $h$ on $K_M^{-1}$ defines the function $H=h(\sigma,\sigma)=|v_1\wedge\dots\wedge v_n|^2_h$. Its properties are controlled by the curvature form $\rho$, as above.

The key point is that, since the vectors $v_i$ are tangent to the canonical $n$-tori $L$, the same function reappears when we calculate the density of the $J$-volume, restricted to such tori: 
\begin{align*}
\vol_J&=\sqrt{H}\,dy\simeq \sqrt{H}\,d\theta\\
&=\sqrt{\det h(\partial y_i,\partial y_j)}\,dy\simeq\sqrt{\det h_{i\bar j}}\,d\theta,
\end{align*}
where the bottom expression holds if $h$ is induced from $M$. If $h$ is $\T^n$-invariant then $H$ is independent of the variable $y$.

\begin{thm}\label{thm:TR_convexity}
Set $M:=\C^n/2\pi i\Z^n\simeq(\C^*)^n$. Let $h$ denote a Hermitian metric on $K_M^{-1}$ (induced for example by a metric on $M$) and $\rho$ the corresponding curvature.

If $\rho\leq 0$ then the $J$-volume functional $\Vol_J$, evaluated on the canonical $n$-tori, is convex.

If the metric is $\T^n$-invariant then the condition $\rho\leq 0$ ($\rho\geq 0$) is equivalent to the convexity (concavity) of $\log\Vol_J$, evaluated on the canonical $n$-tori.
\end{thm}
\begin{proof}The curvature assumption implies that $-i\partial\bar\partial \log H$ is non-positive so $\log H$, thus $\frac{1}{2}\log H=\log\sqrt{H}$, is PSH. It follows that $\sqrt{H}=\exp\circ (\log\sqrt{H})$ is also PSH. 

Applying Proposition \ref{prop:G} we find that $\Vol_J(r)=\iint \sqrt{H}\,d\theta$ is convex with respect to the variable $\log r$, ie $\Vol_J(x)=\iint \sqrt{H}\,dy$ is convex with respect to the variable $x$.

If $H$ is independent of $y$ we can modify this proof as follows. As before, the non-positive curvature condition means that $\log\sqrt{H}$ is PSH. Invariance implies this is equivalent to the convexity of $x\mapsto\iint\log\sqrt{H}\,dy=(2\pi)^n\log\sqrt{H}$, thus of $x\mapsto \log\Vol_J(x)=\log((2\pi)^n\sqrt{H})$. 

The situation for non-negative curvature is analogous.
\end{proof}

Notice that, contrary to the convex case, concavity of $\log\Vol_J$ does not necessarily imply concavity of $\Vol_J$: an example is provided by the function $-x^2=\log(e^{-x^2})$. 

\ 

\noindent\textit{Remark. }In the $\T^n$-invariant case, the function $\log\sqrt H$ (thus $\sqrt H$) descends to the orbit space $\R^n$ and the theorem characterizes its convexity in terms of the Ricci curvature. This is somewhat analogous to the already-mentioned relationship between the Ricci curvature and the density of the ambient volume form $\vol_M$ along geodesics in $M$.

\ 

\noindent\textit{Remark. }The space of canonical TR $n$-tori in $\C^n/2\pi i\Z^n\simeq(\C^*)^n\subseteq M$ is geodesically connected in the sense of \cite{LP2}. Theorem \ref{thm:TR_convexity} is thus related to Theorem 5.10 in \cite{LP2}, where computations are however performed in terms of the Levi-Civita connection.

\ 

\noindent\textit{Remark. }Theorem \ref{thm:TR_convexity} shows that TR submanifolds are well-positioned to encode the geometry of the Ricci curvature. In particular, we obtain the following interpretation of non-positive Ricci curvature.

Let $M$ be a K\"ahler manifold and $p\in M$. Assume $Ric\leq0$ in a neighbourhood of $p$. Given any set of local coordinates, we can assume the domain is $\T^n$-invariant. Theorem \ref{thm:TR_convexity} shows that the $J$-volume of the canonical $n$-tori defined by these coordinates is convex, and Proposition \ref{prop:monotonicity} shows that it is monotone with respect to each $r_i$. When any $r_i=0$ the torus collapses to a lower dimensional submanifold and the $J$-volume is zero. 

\section{Hamiltonian actions}\label{s:symplecticgeometry}
Recall the following notion. Let $(M,\omega)$ be a symplectic manifold endowed with a right $G$-action. Let $\mathfrak{g}$ denote the corresponding Lie algebra. Given any $X\in\mathfrak{g}$, let $\tilde{X}$ denote the corresponding fundamental vector field on $M$. We say that the action is \textit{Hamiltonian} if it admits a \textit{moment map}, ie a map $\mu:M\rightarrow\mathfrak{g}^*$ such that (i) for all $X\in\mathfrak{g}$, $d\mu\cdot X=\omega(\tilde{X},\cdot)$, (ii) $\mu$ is equivariant with respect to the (right) coadjoint action of $G$ on the dual Lie algebra. In particular, (i) implies that $\mathcal{L}_{\tilde X}\omega\equiv 0$ so $\omega$ is $G$-invariant, while (ii) implies that $\mu\cdot [X,Y]\equiv \omega(\tilde X,\tilde Y)$.

Here we are interested in the case $G:=\T^n$: it is Abelian so the coadjoint action is trivial, $\mu$ is invariant and the $\T^n$-orbits are Lagrangian. In this setting we can reformulate Theorem \ref{thm:TR_convexity} as follows.

\begin{thm}\label{thm:Lagr_convexity}
Let $M:=\C^n/2\pi i\Z^n\simeq(\C^*)^n$ be endowed with a K\"ahler metric such that the $\T^n$-action is Hamiltonian. Let $\Vol$ denote the standard volume functional restricted to the set of canonical $n$-tori in  $\C^n/2\pi i\Z^n\simeq(\C^*)^n\subseteq M$. Then:
\begin{itemize}
\item The condition $\Ric\leq 0$ is equivalent to the convexity of $\log\Vol$. 

It implies that $\Vol$ is convex.
\item The condition $Ric\geq 0$ is equivalent to the concavity of $\log\Vol$.

It implies that $1/\Vol$ is convex.
\end{itemize}
In particular, if $\Ric=0$ then $\log \Vol$ is linear. 
\end{thm}

\noindent\textit{Example. }Consider the flat metric on $\C^*$. The $\Sph^1$-orbits have length $2\pi r$, so setting $r=e^t$ we find that $\log\Vol$ is linear wrt the variable $t$, as claimed.

\ 

\noindent\textit{Example. }Recall that, as a complex manifold, $(\C^*)^n$ is a Stein space so its higher Dolbeault cohomology groups vanish. In particular, $H^{1,1}_{\bar\partial}((\C^*)^n)=0$. Assume $\omega$ is $d$-exact. It can then be written as $\omega=d(\alpha+\bar\alpha)$, with $\alpha\in \Lambda^{1,0}((\C^*)^n)$ and $\bar\partial\bar\alpha=0$. We can now solve the problem $\bar\partial h=\bar\alpha$, ultimately writing $\omega=i\partial\bar\partial g$, for some real function $g$. 

Recall that $i\partial\bar\partial$ can alternatively be written as $dd^c$, where $d^cg:=-dg\circ J$. Let us further assume $\omega$ is $\T^n$-invariant. We can then choose $g$ to be invariant and set
\begin{equation*}
\mu:M\rightarrow (\R^n)^*, \ \ \mu\cdot X:=-\d^c g (\tilde{X}).
\end{equation*}
The $\T^n$-invariance and Cartan's formula show that 
\begin{equation*}
0=\mathcal{L}_{\tilde X}d^c g=d\iota_{\tilde X}d^c g+\iota_{\tilde X}dd^c g=-d(\mu\cdot X)+\iota_{\tilde X}\omega,
\end{equation*}
This shows that $d(\mu\cdot X)(\cdot)=\omega(\tilde X,\cdot)$. It is also invariant, so the $\T^n$-action is Hamiltonian. 

\ 

The above result can be applied to characterize positive Ricci curvature and to study the existence/uniqueness of minimal submanifolds, ie of critical points of the volume functional.

\begin{cor}\label{cor:toric_convexity}
Let $M$ be a smooth toric variety endowed with a K\"ahler structure such that the $\T^n$-action is Hamiltonian. Then, on the open orbit, $\Ric>0$ iff $\log\Vol$ is strictly concave. 

In this case $1/\Vol$ is strictly convex and, when $M$ is compact, $\Vol$ has a unique critical point corresponding to a global maximum. It is a minimal Lagrangian $n$-torus.
\end{cor}
\begin{proof} We have already proved the first statement. The orbits in the compactification of $(\C^*)^n$ have strictly lower dimension so their volume is zero. Compactness implies that $\Vol$ has a maximum. The critical points of $\Vol$ and $\log\Vol$ coincide. Concavity implies the latter has a unique such point. In theory the corresponding Lagrangian submanifold is critical only within the class of $\T^n$-orbits, but by the Palais' symmetry principle it is automatically critical with respect to all variations so it is minimal.
\end{proof}

The statement concerning the uniqueness of critical points first appeared in \cite{OSD}, where it was proved using a different set of coordinates (action-angle coordinates). In that coordinate system, however, $\Vol$ has no convexity properties.

\ 

\noindent\textit{Example. }Let $M$ be a compact Fano smooth toric variety, ie $c_1(M)>0$. Symmetrization preserves cohomology classes, so Yau's theorem shows that any $\T^n$-invariant positive form in $c_1(M)$ is the Ricci form of some $\T^n$-invariant  K\"ahler metric representing $c_1(M)$. Futaki \cite{Futaki} showed that, given any such metric, the $\T^n$-action is Hamiltonian. We can thus apply Corollary \ref{cor:toric_convexity}. Invariant K\"ahler-Einstein metrics \cite{WangZhu} are a special case.

\ 
 
Appropriate boundary conditions, ensuring the properness of the volume functional, would analogously prove the existence and uniqueness of minimal orbits in the case $\Ric<0$.

\bibliographystyle{amsplain}
\bibliography{PSH_biblio}

\begin{bibdiv}
\begin{biblist}

\bib{Abreu}{incollection}{
      author={Abreu, Miguel},
       title={K\"{a}hler geometry of toric manifolds in symplectic
  coordinates},
        date={2003},
   booktitle={Symplectic and contact topology: interactions and perspectives
  ({T}oronto, {ON}/{M}ontreal, {QC}, 2001)},
      series={Fields Inst. Commun.},
      volume={35},
   publisher={Amer. Math. Soc., Providence, RI},
       pages={1\ndash 24},
      review={\MR{1969265}},
}

\bib{Brody}{article}{
      author={Brody, Robert},
       title={Compact manifolds and hyperbolicity},
        date={1978},
        ISSN={0002-9947},
     journal={Trans. Amer. Math. Soc.},
      volume={235},
       pages={213\ndash 219},
         url={https://doi-org.bibliopass.unito.it/10.2307/1998216},
      review={\MR{470252}},
}

\bib{Demailly}{article}{
      author={Demailly, Jean-Pierre},
       title={Complex analytic and differential geometry},
     journal={author's website},
}

\bib{Futaki}{article}{
      author={Futaki, Akito},
       title={The {R}icci curvature of symplectic quotients of {F}ano
  manifolds},
        date={1987},
        ISSN={0040-8735},
     journal={Tohoku Math. J. (2)},
      volume={39},
      number={3},
       pages={329\ndash 339},
         url={https://doi-org.bibliopass.unito.it/10.2748/tmj/1178228281},
      review={\MR{902573}},
}

\bib{Gilligan}{article}{
      author={Gilligan, Bruce},
      author={Miebach, Christian},
      author={Oeljeklaus, Karl},
       title={Homogeneous {K}\"{a}hler and {H}amiltonian manifolds},
        date={2011},
        ISSN={0025-5831},
     journal={Math. Ann.},
      volume={349},
      number={4},
       pages={889\ndash 901},
         url={https://doi-org.bibliopass.unito.it/10.1007/s00208-010-0546-y},
      review={\MR{2777037}},
}

\bib{HLW}{article}{
      author={Heier, Gordon},
      author={Lu, Steven S.~Y.},
      author={Wong, Bun},
       title={K\"{a}hler manifolds of semi-negative holomorphic sectional
  curvature},
        date={2016},
        ISSN={0022-040X},
     journal={J. Differential Geom.},
      volume={104},
      number={3},
       pages={419\ndash 441},
  url={http://projecteuclid.org.bibliopass.unito.it/euclid.jdg/1478138548},
      review={\MR{3568627}},
}

\bib{LP2}{article}{
      author={Lotay, Jason~D.},
      author={Pacini, Tommaso},
       title={Complexified diffeomorphism groups, totally real submanifolds and
  {K}\"{a}hler-{E}instein geometry},
        date={2019},
        ISSN={0002-9947},
     journal={Trans. Amer. Math. Soc.},
      volume={371},
      number={4},
       pages={2665\ndash 2701},
         url={https://doi-org.bibliopass.unito.it/10.1090/tran/7421},
      review={\MR{3896093}},
}

\bib{LP1}{article}{
      author={Lotay, Jason~D.},
      author={Pacini, Tommaso},
       title={From {L}agrangian to totally real geometry: coupled flows and
  calibrations},
        date={2020},
        ISSN={1019-8385},
     journal={Comm. Anal. Geom.},
      volume={28},
      number={3},
       pages={607\ndash 675},
         url={https://doi-org.bibliopass.unito.it/10.4310/CAG.2020.v28.n3.a5},
      review={\MR{4124139}},
}

\bib{Maccheroni}{article}{
      author={Maccheroni, Roberta},
       title={Complex analytic properties of minimal {L}agrangian
  submanifolds},
        date={2020},
        ISSN={1527-5256},
     journal={J. Symplectic Geom.},
      volume={18},
      number={4},
       pages={1127\ndash 1146},
         url={https://doi-org.bibliopass.unito.it/10.4310/JSG.2020.v18.n4.a6},
      review={\MR{4174296}},
}

\bib{OSD}{article}{
      author={Oliveira, Goncalo},
      author={Sena-Dias, Rosa},
       title={Minimal lagrangian tori and action-angle coordinates},
     journal={arxiv},
}

\bib{Sommese}{article}{
      author={Sommese, Andrew~John},
       title={Extension theorems for reductive group actions on compact
  {K}aehler manifolds},
        date={1975},
        ISSN={0025-5831},
     journal={Math. Ann.},
      volume={218},
      number={2},
       pages={107\ndash 116},
         url={https://doi-org.bibliopass.unito.it/10.1007/BF01370814},
      review={\MR{393561}},
}

\bib{WangZhu}{article}{
      author={Wang, Xu-Jia},
      author={Zhu, Xiaohua},
       title={K\"{a}hler-{R}icci solitons on toric manifolds with positive
  first {C}hern class},
        date={2004},
        ISSN={0001-8708},
     journal={Adv. Math.},
      volume={188},
      number={1},
       pages={87\ndash 103},
         url={https://doi-org.bibliopass.unito.it/10.1016/j.aim.2003.09.009},
      review={\MR{2084775}},
}

\bib{WuYau}{article}{
      author={Wu, Damin},
      author={Yau, Shing-Tung},
       title={Negative holomorphic curvature and positive canonical bundle},
        date={2016},
        ISSN={0020-9910},
     journal={Invent. Math.},
      volume={204},
      number={2},
       pages={595\ndash 604},
         url={https://doi-org.bibliopass.unito.it/10.1007/s00222-015-0621-9},
      review={\MR{3489705}},
}

\end{biblist}
\end{bibdiv}

\end{document}